\documentclass[sn-mathphys]{sn-jnl}

\jyear{2023}%

\usepackage{cancel}
\usepackage{hyperref}

\theoremstyle{thmstyleone}%
\newtheorem{theorem}{Theorem}
\newtheorem{lemma}[theorem]{Lemma}
\newtheorem{corollary}[theorem]{Corollary}
\newtheorem{proposition}[theorem]{Proposition}%

\theoremstyle{thmstyletwo}%

\theoremstyle{thmstylethree}%
\newtheorem{definition}{Definition}%

\begin{document}

\title[Characterising Solutions of Anomalous Cancellation]{Characterising Solutions of Anomalous Cancellation}


\author*[1]{\fnm{Satvik} \sur{Saha}}\email{ss19ms154@iiserkol.ac.in}

\author[2]{\fnm{Sohom} \sur{Gupta}}\email{sg19ms141@iiserkol.ac.in}

\author[1]{\fnm{Sayan} \sur{Dutta}}\email{sd19ms148@iiserkol.ac.in}

\author[1]{\fnm{Sourin} \sur{Chatterjee}}\email{sc19ms170@iiserkol.ac.in}

\affil[1]{\orgdiv{Department of Mathematics and Statistics}, \orgname{Indian Institute of Science Education and Research, Kolkata}, \orgaddress{\country{India}}}

\affil[2]{\orgdiv{Department of Physical Sciences}, \orgname{Indian Institute of Science Education and Research, Kolkata}, \orgaddress{\country{India}}}


\abstract{
    Anomalous cancellation of fractions is a mathematically inaccurate method where cancelling the common digits of the numerator and denominator correctly reduces it. While it appears to be \emph{accidentally} successful, the property of anomalous cancellation is intricately connected to the number of digits of the denominator as well as the base in which the fraction is represented. Previous work have been mostly surrounding three digit solutions or specific properties of the same. This paper seeks to get general results regarding the structure of numbers that follow the cancellation property (denoted by $P^*_{\ell; k}$) and an estimate of the total number of solutions possible in a given base representation. In particular, interesting properties regarding the \emph{saturation} of the number of solutions in general and $p^n$ bases (where $p$ is a prime) have been studied in detail.%
}




\maketitle

\section{Introduction}\label{sec1}





A zeal for some interesting mathematical problems brought us to a very peculiar problem, a quest to find all odd digit integers $[a_1a_2\ldots a_{2k+1}]$ (all $a_i$'s are digits, $a_1\neq 0$) such that the following property holds.
\[[a_1a_2\ldots a_k]\cdot [a_{k+1}a_{k+2}\ldots a_{2k+1}] = [a_1a_2\ldots a_{k+1}]\cdot [a_{k+2}a_{k+2}\ldots a_{2k+1}]\]

An elementary example is the number 164, which has the property as shown $1\times 64 = 16\times 4$. Similarly $24\times 996 = 249\times 96$ implies that $24996$ also fits our problem requirement. Now the question is, can one find a way to generate all such numbers? A brute force algorithm always works, but it is never satisfying to leave things at that -- a proof of being a mathematics aspirant. This prompted us to go through existing literature which, though scarce, hide a gold mine of information. Adding on to previous work, we arrived at several interesting results that demand a place in this paper. The paper has been structured in a format that the general reader can be presented with all the beautiful results, while the seasoned readers can move to the appendix to get a flavour of the relevant proofs. All proofs use elementary number theory techniques and hence this paper is directed at undergraduate students and beyond. \\

The anomalous cancellation property for fractions of the form $[ab]_B / [bc]_B$, where
the numerator and denominator are two digit integers in base $B$, has been studied
extensively \cite{boas}. Finding such fractions amounts to solving the following
Diophantine equation in $a, b, c$. \[
    (aB + b)c = a(bB + c).
\] 

We shall quickly state some nice results in literature, to keep readers up to date.

Apart from trivial solutions such as $a = b = c$ or $b = c = 0$, it has been shown
that $b$ is the largest digit. With this, the transformation $(a, b, c) \mapsto 
(b - c, b, b - a)$ is an involution on the set of non-trivial solutions \cite{yt}.
Ekhad \cite{ekh} has examined a much more general anomalous cancellation property
where a digit is allowed to be cancelled from anywhere in the numerator and 
denominator. Fixing the base, denominator, and the digit indices used during 
cancellation reduces the problem of finding all solutions to a linear Diophantine
equation, which can be solved easily.

It should be good to go dive in our work now. For no future confusion, we shall take some time to sort out the notations used frequently. All other symbols have their usual meaning.

\paragraph{Notation.} Let $B \geq 2$ be an integer base, and let $x_1, x_2, \dots, x_k$ be
integer digits, with each $0 \leq x_i < B$. Then, we denote their concatenation
as \[
    x = [x_1x_2\dots x_k] \;\equiv\; \sum_{i = 1}^k x_i B^{k - i}.
\] In our problem, we consider natural numbers $N_{\ell; k}$ in base $B$, 
where $\ell, k \geq 1$. Its digits are denoted as \[
    N_{\ell; k} = [a_1a_2\dots a_\ell \,b\, c_1c_2\dots c_k] = [a b c],
\] with blocks \[
    a = [a_1a_2\dots a_\ell], \qquad
    c = [c_1c_2\dots c_k].
\] Note that we may write \[
    N_{\ell; k} = aB^{k + 1} + bB^k + c.
\]

\begin{definition}
    We say that the number $N_{\ell; k}$ has property $P_{\ell; k}$,
    or is a solution of $P_{\ell; k}$, if \[
        [a_1a_2 \dots a_\ell\, b] \times [c_1c_2 \dots c_k] =
        [a_1a_2 \dots a_\ell] \times [b\, c_1c_2 \dots c_k].
    \] Equivalently, \[
        \frac{[a_1a_2\dots a_\ell\, \cancel{b}]}{
        [\cancel{b}c_1c_2\dots c_k]} = \frac{[a_1a_2\dots a_\ell]}{
        [c_1c_2\dots c_k]}
    \] which means that the digit 'b' can be anomalously cancelled in the LHS fraction to \emph{correctly} reduce to the RHS equation. This can also be expressed as the homogeneous Diophantine equations \[
        (aB + b)c = a(bB^k + c), \qquad \frac{1}{a} + \frac{B - 1}{b} = \frac{B^k}{c}.
    \] In the special case $\ell = k$, we abbreviate $N_{k; k} \equiv N_k$ and 
    $P_{k; k} \equiv P_k$.
\end{definition}

\paragraph{Convention.}
Certain types of solutions of $P_{\ell; k}$ exist given any base
$B$. These are of the following forms. \begin{itemize}
    \item $[dd\dots dd]$.
    \item $[a_1a_2\dots a_k\,0\,0\dots 0]$ (in particular, any two of $a, b, c = 0$.)
\end{itemize}
We term these as \emph{trivial solutions}.
In our discussion, we ignore such solutions, and focus on the case $\ell = k$,
i.e.\ we are searching for non-trivial solutions to the $P_k$ problem.

\begin{definition}
    We say that $N_{\ell; k}$ has property $P^*_{\ell; k}$ if it has $P_{\ell; k}$
    non-trivially, i.e\ no two of $a, b, c = 0$ and not all digits $a_i = b = c_i$.
\end{definition}

\section{Characterizing how solutions look like.} \label{section:patterns}

The first important observation is that solutions of $P_{\ell; k}$ can be
extended very simply to solutions of $P_{\ell + 1; k + 1}$, by replacing
the central digit $b$ with the digits $[b b b]$. This is a special case
of the Pumping Lemma described in \cite{gold}.

\begin{proposition} \label{proposition:extension}
    Let $B$ be an arbitrary integer base, and let \[
        N = [a_1a_2\dots a_\ell \,b\, c_1c_2\dots c_k], \qquad
        N^+ = [a_1a_2\dots a_\ell \,b\: b\: b\, c_1c_2\dots c_k].
    \] Then, $N$ has property $P_{\ell; k}$ if and only if $N^+$ has property
    $P_{\ell + 1; k + 1}$. We say that $N^+$ is an \emph{extension} of $N$.
\end{proposition}
\begin{proof}
    The number $N^+$ has property $P_{\ell + 1; k + 1}$ precisely when \[
        (aB^2 + bB + b)(bB^k + c) = (aB + b)(bB^{k + 1} + bB^k + c),
    \] which after expanding and cancelling reduces to \[
        (aB + b)cB = aB(bB^k + c), \qquad (aB + b)c = a(bB^k + c),
    \] which is precisely the statement that $N$ has property $P_{\ell; k}$.
\end{proof}

This shows that the number of solutions of $P^*_k$
in a particular base $B$ cannot decrease with increasing $k$. A natural question is
whether the number of such solutions gets arbitrarily large; can we keep producing
solutions of the $P_k$ problem that aren't merely extensions of old ones? To answer
this, we show that the digits of any such solution must obey very rigid rules,
culminating in the following.

\begin{theorem} \label{theorem:structure}
    Let $B$ be an arbitrary integer base and let $N_k$ have property 
    $P^*_k$.
    Then, the digits must satisfy the following constraints.
    \begin{enumerate}
        \item $a_1 < B / 2$.
        \item $b = c_1 = c_2 = \dots = c_{k - 1} > c_k > 1$.
        \item $\gcd(c_k, B) > 1$.
        \item $\gcd(a_k - b, B) > 1$ if $a_k \neq b$.
    \end{enumerate}
\end{theorem}

The proof of Theorem~\ref{theorem:structure} will follow by combining
Corollary~\ref{corollary:last_digit} and Lemmas~\ref{lemma:leading_digit},
\ref{lemma:last_block}, \ref{lemma:gcds}, which we prove later.

An immediate consequence is that all solutions of $P^*_k$ look like \[
    [a_1a_2\dots a_k \,b\, b\dots b c_k].
\] This means that the blocks are of the form \[
    a = \frac{bM}{B} + \frac{bc_kB^{k - 1}}{bB - (B - 1)c_k}, \qquad 
    c = bM + c_k, \qquad
    M = \frac{B^k - B}{B - 1}.
\]

Theorem~\ref{theorem:structure} guarantees that when the base $B$ is even, plugging
in $b = B - 1$, $c_k = B - 2$ corresponds to the largest solution of $P^*_k$, \[
    a = \frac{1}{2}B^k - 1, \qquad b = B - 1, \qquad c = B^k - 2.
\] For example, in base $B = 10$, the largest solution of $P^*_3$ is $4999998$.

Theorem~\ref{theorem:structure} also gives us a nice criterion for determining whether
solutions of $P^*_k$ exist in the first place.

\begin{theorem} \label{theorem:prime_composite}
    The $P^*_k$ problem admits solutions if and only if the base $B$ is composite.
\end{theorem}
\begin{proof}
    If the base $B = mn$ for $m, n > 1$, then $N_k = [abc]$ with \[
        (a, b, c) = (mB^{k - 1} - 1, B - 1, B^k - n)
    \] has property $P^*_k$.
    Conversely, if the base $B = p$ for prime $p$, then any solution $N_k$ of 
    $P^*_k$ must satisfy $0 < c_k < p$ and $\gcd(c_k, p) > 1$ simultaneously by
    Proposition~\ref{theorem:structure}, a contradiction.
\end{proof}

\section{Saturation of solutions.}

In Section~\ref{section:patterns}, we have shown that every candidate solution of
$P^*_k$ can be completely written in terms of the two single digits $b$ and $c_k$.
Here, we say that $(b, c_k)$ \emph{generates} such a solution of $P^*_k$.
Figure~\ref{fig:solution_space} visualizes the solution space of $P^*_{101}$ in
base $B = 126$ via such generating tuples.
Thus, finding solutions amounts to testing pairs $(b, c_k)$, with $1 < c_k < b < B$.
Importantly, this search space depends only on $B$, not $k$(!)
Thus, for a given base $B$, the number of solutions of $P^*_k$ \emph{saturates}.

\begin{proposition} \label{proposition:number_pk}
    Let $B$ be an arbitrary integer base. There are at most \[
        \frac{(B - 2)(B - 3)}{2}
    \] solutions of $P^*_k$, for any $k \geq 1$.
\end{proposition}
\begin{proof}
    For each $b$, we have $2 \leq c_k < b$, i.e.\ $b - 2$ choices. Putting 
    $3 \leq b < B$, we obtain a total of $1 + 2 + \dots + (B - 3) = (B - 2)(B - 3) / 2$
    candidate tuples $(b, c_k)$.
\end{proof}

The above estimate is clearly very crude: we have only used the relation
$b = c_{i < k} > c_k$.
Indeed, numerical evidence suggests that a much sharper bound can be established.
To illustrate this, the only solutions of $P^*_k$ for base $B = 10$ are extensions
of the following (in order of appearance with increasing $k$).
\[
    164, \;
    265, \;
    195, \;
    498, \;
    21775, \;
    24996, \;
    1249992, \;
    340277776.
\]

We next describe when the full gamut of solutions for a base $B$ is achieved.

\begin{theorem} \label{theorem:saturation_point}
    Let $B$ be an arbitrary integer base. Then, the number of solutions of 
    $P^*_k$ become constant beyond \[
        k = \max\{5,\,2\log_2{(B-1)} + 2\}.
    \]
\end{theorem}

This means that for $k$ beyond the above \emph{saturation point}, all solutions of 
$P^*_k$ are mere extensions of old ones. A proof of
Theorem~\ref{theorem:saturation_point} is supplied in Appendix~\ref{section:estimate_saturation}.

\begin{figure}
    \centering
    \includegraphics[scale = 0.5]{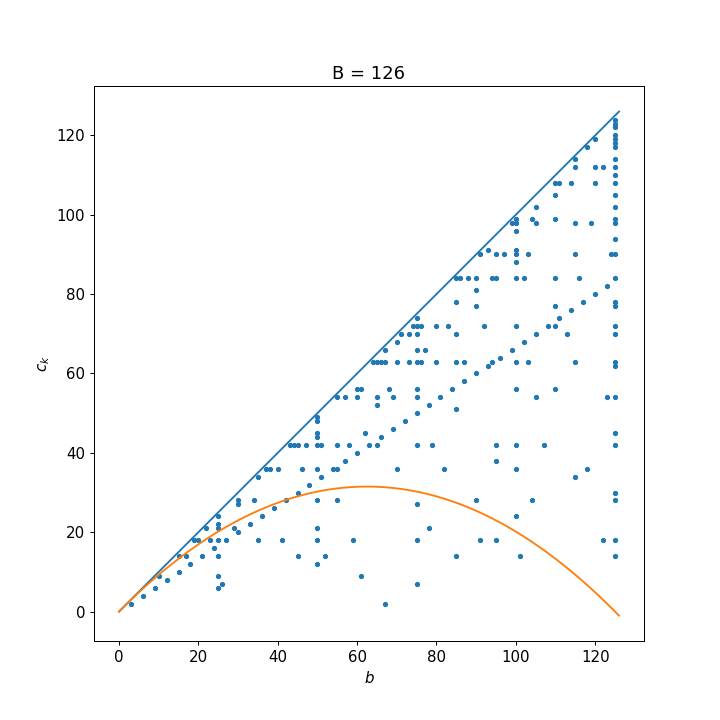}
    \caption{Generating tuples for $P^*_{\ell; k}$ with $B = 126$, $k = 101$, represented by blue dots. Points above the orange curve generate solutions of $P^*_k$ with the candidate block $a \geq B^{k - 1}$. Points underneath the orange curve generate solutions of $P^*_{\ell; k}$ with $\ell < k$.}
    \label{fig:solution_space}
\end{figure}

The bound on $k$ described in Theorem~\ref{theorem:saturation_point} is especially
bad for one family of bases in particular. When $B = p^n$ for prime $p$, $n \geq 2$,
the solutions of $P^*_k$ saturate in the very first step, $k = 1$.

\begin{theorem} \label{theorem:base_pn}
    Let $B = p^n$ for prime $p$. Then, all solutions of $P^*_k$ are extensions of
    those of $P^*_1$.
\end{theorem}
\begin{proof}
    By Lemma~\ref{lemma:pn_extension}, any solution $N_k$ of $P^*_k$ looks like \[
        N_k = [a_1b\dots b\,b\,b\dots bc_k].
    \]
    Proposition~\ref{proposition:extension} can be used $k - 1$ times to reduce this to
    the solution $[a_1bc_k]$ of $P^*_1$.
\end{proof}

For example, the only solutions of $P^*_k$ for base $B = 3^2$ are of the form \[
    14\dots 43, \quad 28\dots 86.
\]

\section{Discussion.}

Our main goal was to characterize all solutions of the anomalous cancellation problem $P^*_k$,
which still seems to be a far way off. However, we have successfully characterized solutions in
certain special bases, and given rudimentary bounds on their number. Our major findings are
summarized below.

\begin{itemize}
    \item There are no solutions of $P^*_k$ for prime bases.
    \item The only solutions of $P^*_k$ for prime-power bases $p^n$ are extensions of solutions
    of $P^*_1$.
    \item There are at least as many solutions of $P^*_k$ for composite bases $B$ as there
    are non-trivial factors of $B$. These are of the form $[abc]$ with \[
        (a, b, c) = (mB^{k - 1} - 1, B - 1, B^k - n)
    \] for $B = mn$, $m, n > 1$, $M = (B^k - B)/(B - 1)$.
\end{itemize}

Some very interesting and peculiar observations were made from numerical solutions; while
these haven't been proved in this paper, there is a lot of scope for future expansion.

\begin{itemize}
    \item Suppose in base $B$, we have no new non-trivial solutions
    (except extensions) in $(2k+1)$ digits. We have observed that there would be no new solutions
    in $(2k+3)$ digits. In other words, as long as the number of solutions do not saturate, they will
    keep on increasing.
    \item Currently we have a quadratic bound on the total number of solutions possible in a given base $B$,
    but observed solution counts are much smaller than that.
    \item An interesting result discussed is that composite bases are guaranteed to have
    solutions that are of the form $(a, B-1, c)$. Therefore, finding out no solutions of
    this form guarantees primality. While the current search space is only as good as a brute
    force method, some better ideas relating $b$ and $c_k$ might lead to a
    much faster primality test.
\end{itemize}

\backmatter


\bmhead{Acknowledgments}

The authors would like to thank Prof.\ Soumya Bhattacharya for his careful reading of the manuscript and many helpful discussions on the topic of this paper.



\begin{appendices}

\section{Trivial solutions.}

This section focuses on trivial solutions of $P_{\ell; k}$. We supply
a few criteria which can be used to swiftly identify candidate solutions
as trivial, based on a subset of their digits.

\begin{lemma} \label{lemma:divisibility}
    Let $N_{\ell; k}$ have property $P_{\ell; k}$. Then, \begin{enumerate}
        \item $a \mid bc$.
        \item $b \mid ac(B - 1)$.
        \item $c \mid abB^k$.
    \end{enumerate}
\end{lemma}
\begin{proof}
    Using the fact that $N$ has the $P_{\ell; k}$ property, we have \[
        (aB + b)c = a(bB^k + c).
    \] This can be rewritten by collecting each of $a, b, c$ successively on one side,
    giving \[
        a(bB^k - (B - 1)c) = bc, \qquad
        b(aB^k - c) = ac(B - 1), \qquad
        c(a(B - 1) + b) = abB^k,
    \] from which the desired rules follow.
\end{proof}

\begin{lemma} \label{lemma:trivial_1}
    Let $N_{\ell; k}$ have property $P_{\ell; k}$.
    If any one of $a, b, c = 0$, then at least one of the others is also $0$, i.e.\ $N$
    is a trivial solution for the $P_{\ell; k}$ problem.
\end{lemma}
\begin{proof}
    This follows from the divisibility conditions in Lemma~\ref{lemma:divisibility}.
\end{proof}

\begin{lemma} \label{lemma:trivial_2}
    Let $N_k$ have property $P_k$ and at least one of the following hold.
    \begin{enumerate}
        \item $a = c$.
        \item $a_i = b$ for all $1 \leq i \leq k$.
        \item $c_i = b$ for all $1 \leq i \leq k$.
    \end{enumerate}
    Then, all the digits $a_i = c_i = b$, i.e.\ $N$ is a trivial solution for 
    the $P_k$ problem.
\end{lemma}
\begin{proof}
    Let $N$ have property $P_k$, whence $(aB + b)c = a(bB^k + c)$.
    Denote \[
        I = [\underbrace{1 1 \dots 1}_k] = \frac{B^k - 1}{B - 1}.
    \]
    \begin{enumerate}
        \item Putting $a = c$, \[
            a = \frac{B^k - 1}{B - 1}\cdot b, \qquad
            \sum_{i = 1}^k a_iB^{k - i} = \sum_{i = 1}^k bB^{k - i}.
        \] By the uniqueness of representation in the base $B$, each $a_i = b$.

        \item Putting $a = bI$, \[
            c = \frac{abB^k}{a(B - 1) + b} = \frac{b^2I B^k}{b(I(B - 1) +
            1)} = bI.
        \]

        \item Putting $c = bI$, \[
            a = \frac{bc}{bB^k - (B - 1)c} = \frac{b^2I}{bB^k - b(B - 1)I} = 
            \frac{b^2I}{b(B^k - (B^k - 1))} = bI. 
        \]
    \end{enumerate}
\end{proof}

\section{Uneven blocks.}

Although we primarily deal with solutions of $P^*_k$ in this paper, it
is necessary to make a short detour and examine a few aspects of the more
general $P^*_{\ell; k}$ problem in order to prove Theorem~\ref{theorem:structure}.

\begin{lemma} \label{lemma:reduce_right}
    Let $B$ be an arbitrary integer base and let $N_{\ell; k}$ have
    property $P_{\ell; k}$. If $c_k = 0$, then the number \[
        N^- = [a_1a_2\dots a_\ell \,b\, c_1c_2\dots c_{k - 1}]
    \] has property $P_{\ell; k - 1}$
\end{lemma}
\begin{proof}
    Denote \[
        c' = [c_1c_2\dots c_{k - 1}], \qquad
        c = c'B + c_k = c'B.
    \] Since $N$ has property $P_{\ell; k}$, we have $(aB + b)c = a(bB^k + c)$, hence \[
        (aB + b)c'B = a(bB^k + c'B) \qquad
        (aB + b)c' = a(bB^{k - 1} + c'),
    \] which is precisely the statement that $N^-$ has property $P_{\ell; k - 1}$.
\end{proof}

\begin{lemma} \label{lemma:block_size}
    Let $B$ be an arbitrary integer base and let $N_{\ell; k}$
    have property $P^*_{\ell; k}$. Then, $\ell \leq k$. 
    In other words, there are no solutions of $P^*_{\ell; k}$ when $\ell > k$.
\end{lemma}
\begin{proof}
    If $N_{\ell; k}$ has property $P^*_{\ell; k}$, then $a, b, c > 0$, \[
        (aB + b)c = a(bB^k + c), \qquad
        a(bB^k + c - Bc) = bc.
    \] Put $bB^k + c - Bc = bc / a = d$, which is a positive integer. Suppose that 
    $1 \leq d < B$, i.e.\ $d$ is a single digit in base $B$. Expanding 
    $bB^k + c = cB + d$ gives us \[
        bB^k + c_1B^{k - 1} + \dots + c_{k - 1}B + c_k = c_1B^k + c_2B^{k - 1} + \dots 
        + c_kB + d.
    \] By the uniqueness of representation of integers in the base $B$, we equate
    the coefficients $b = c_1$, $c_1 = c_2$, \dots, $c_{k - 1} = c_k$, $c_k = d$; 
    specifically, $b = d$. Thus, $a = bc / d = c$, hence the solution is trivial by
    Lemma~\ref{lemma:trivial_2}.

    This means that for $N$ to be a non-trivial solution, we must have $d \geq B$. Now, 
    $0 < b < B$ and $0 < c < B^k$, hence \[
        a = \frac{bc}{d} < \frac{B \cdot B^k}{B} = B^k.
    \] This shows that $a$ can have at most $k$ digits, hence $\ell \leq k$.
\end{proof}

\begin{corollary} \label{corollary:last_digit}
    If $N_k$ has property $P^*_k$, then $c_k \neq 0$.
\end{corollary}
\begin{proof}
    If $c_k = 0$, we see that \[
        N^- = [a_1a_2\dots a_k \,b\, c_1c_2\dots c_{k - 1}]
    \] has property $P_{k; k - 1}$, and hence must be a trivial solution by
    Lemma~\ref{lemma:block_size}. Furthermore, it must be trivial in the sense
    that one of $a, b, c = 0$; if not, then $a < B^{k - 1}$ from the
    lemma contradicts the fact that $a$ is a $k$-digit number.
    Thus, the original number $N_k$ is also a trivial solution.
\end{proof}

\begin{corollary} \label{corollary:trivial_ratio}
    If $N_k$ has property $P^*_{\ell; k}$, then the integer $d = bc / a \geq B$.
\end{corollary}

The technique used in Lemma~\ref{lemma:block_size} can be employed to obtain
an even sharper bound on the leading block $a$ of a solution $N_k$.

\begin{lemma} \label{lemma:leading_digit}
    Let $B$ be an arbitrary integer base and let $N_k$ have property
    $P^*_k$.
    Then, $a < B^k / 2$. As a result, the leading digit $a_1 < B / 2$.
\end{lemma}
\begin{proof}
    Continuing along the same lines as the proof of Lemma~\ref{lemma:block_size}, 
    suppose that $N_k$ has property $P^*_k$. Then $a, b, c > 0$, \[
        (aB + b)c = a(bB^k + c), \qquad
        a(bB^k + c - Bc) = bc,
    \] and $bB^k + c - Bc = bc / a = d$ is a positive integer. We saw 
    that when $1 \leq d < B$, the solution $N$ is trivial. Furthermore, when 
    $d \geq 2B - 2$, observe that \[
        a = \frac{bc}{d} \leq \frac{(B - 1)\cdot (B^k - 1)}{2B - 2} = \frac{1}{2} (B^k - 1).
    \]

    We now examine the remaining case $B \leq d < 2B - 2$. Setting $d' = d - B$, we have
    $0 \leq d' < B - 2$, i.e\ $d'$ is a single digit in base $B$. Expanding $bB^k + c = cB + d$
    gives us \begin{align*}
        bB^k + c_1B^{k - 1} &+ \dots + c_{k - 1}B + c_k \\
        &= c_1B^k + c_2B^{k - 1} + \dots + c_kB + (B + d').
    \end{align*}
    This implies $B \mid c_k - d'$; but $0 \leq \vert c_k - d'\vert < B$ forcing 
    $c_k = d'$. Since $k \geq 1$, we can subtract $c_k = d'$ and divide $B$, yielding
    \begin{align*}
        bB^{k - 1} + c_1B^{k - 2} &+ \dots + c_{k - 1} \\
        &= c_1B^{k - 1} + c_2B^{k - 2} + \dots + c_{k - 1}B + d' + 1. \tag{$\star$}
    \end{align*}
    Since $d' < B - 2$, the number $d' + 1 < B - 1$ is a single digit, so we 
    can equate coefficients and see that $b = c_1$, $c_1 = c_2$, \dots, 
    $c_{k - 2} = c_{k - 1}$, $c_{k - 1} = d' + 1$, hence $b = c_1 = \dots = c_{k - 1} = 
    d' + 1$. In other words, all the digits of $c + 1$ are exactly $d' + 1$, 
    so if we set \[
        I = [1_1 1_2\dots 1_k] = \frac{B^k - 1}{B - 1},
    \] then $c + 1 = (d' + 1)I$. Then, \[
        a = \frac{bc}{d} = \frac{(d' + 1)[(d' + 1)I - 1]}{B + d'} 
        < \frac{d'\cdot [(B - 1)I - 1]}{2d'} = 
        \frac{1}{2}[(B - 1)I - 1]
    \] hence putting $(B - 1)I = B^k - 1$ gives \[
        a < \frac{1}{2}(B^k - 2). 
    \]
\end{proof}

\section{The trailing block.}

This section deals with the trailing block $c$ of a solution $N_k$, which
can be almost completely described in terms of the central block $b$.

\begin{lemma} \label{lemma:last_block}
    Let $B$ be an arbitrary integer base and let $N_k$ have property 
    $P^*_k$.
    Then, the digits in the last block satisfy $c_k < c_{i < k} = b$.
    In other words, all solutions of $P^*_k$ look like \[
        N = [a_1a_2\dots a_k \,b\, b b\dots b c_k]
    \]
\end{lemma}
\begin{proof}
    As before, suppose that $N_k$ has property $P^*_k$. Then $a, b, c > 0$, \[
        (aB + b)c = a(bB^k + c), \qquad
        a(bB^k + c - Bc) = bc,
    \] and $bB^k + c - Bc = bc / a = d$ is a positive integer.
    By Corollary~\ref{corollary:trivial_ratio}, we have $d \geq B$. Since \[
        d = \frac{bc}{a} <\frac{(B - 1)\cdot B^k}{B^{k - 1}} < (B - 1)B < B^2,
    \] we have $d = qB + d'$, with $0 < q, d' < B$; the fact
    that $q > 0$ follows from $d \geq B$.
    Expanding $bB^k + c = cB + d$, we have \begin{align*}
        bB^k + c_1B^{k - 1} &+ \dots + c_{k - 1}B + c_k \\
            &= c_1B^k + c_2B^{k - 1} + \dots + c_kB + (qB + d'). \tag{$\star$}
    \end{align*}
    We have $B \mid c_k - d'$, forcing $c_k = d'$.

    Consider the case $k = 1$, where our equation now reads $bB + c = cB + d$, hence 
    $bB + d' = d'B + qB + d'$, so $bB = (d' + q)B$.
    Thus, $b = d' + q > d' = c$ as desired.

    Now let $k \geq 2$. Subtracting $c_k = d'$ from both sides of ($\star$) and dividing by $B$ 
    gives \begin{align*}
        bB^{k - 1} + c_1B^{k - 2} &+ \dots + c_{k - 2}B + c_{k - 1} \\
        &= c_1B^{k - 1} + c_2B^{k - 2} + \dots + c_{k - 1}B + d' + q.
    \end{align*}
    Note that $2 \leq d' + q \leq 2B - 2 < 2B$, so expand $d' + q = q'B + r$ for some 
    $0 \leq r < B$, and $q' = 0, 1$. If $q' = 0$,  then \begin{align*}
        bB^{k - 1} + c_1B^{k - 2} &+ \dots + c_{k - 2}B + c_{k - 1} \\
        &= c_1B^{k - 1} + c_2B^{k - 2} +\dots + c_{k - 1}B + r,
    \end{align*}
    hence we can equate coefficients yielding $b = c_1$, $c_1 = c_2$, \dots, $c_{k - 2} =
    c_{k - 1}$, $c_{k - 1} = r = d' + q = c_k + q > c_k$. In other words, all 
    $b = c_{i < k} > c_k$.

    Otherwise, $q' = 1$, and \begin{align*}
        bB^{k - 1} + c_1B^{k - 2} &+ \dots + c_{k - 2}B + c_{k - 1} \\
        &= c_1B^{k - 1} + c_2B^{k - 2} + \dots + (c_{k - 1} + 1)B + r.
    \end{align*}
    This gives $B \mid c_{k - 1} - r$, hence $c_{k - 1} = r$ anyways. Now if $r = B - 1$,
    we would have $d' + q = q'B + r = 2B - 1$; this contradicts $q' + d \leq 2B - 2$.
    Thus, $r \leq B - 2$, so $c_{k - 1} + 1 = r + 1 \leq B - 1$ is a single digit in base $B$.
    Equating coefficients, $b = c_1$, $c_1 = c_2$, \dots, $c_{k - 2} = c_{k - 1} = r$.
    Furthermore, $c_k = d' = (q'B + r) - q = (B - q) + r < r = c_{k - 1}$. Thus, we again have
    $b = c_{i < k} > c_k$.
\end{proof}

\begin{corollary} \label{corollary:ac_blocks}
    Let $B$ be an arbitrary integer base and let $N_k$ have property $P^*_k$.
    Then, \[
        a = \frac{bM}{B} + \frac{bc_kB^{k - 1}}{bB - (B - 1)c_k}, \qquad 
        c = bM + c_k, \qquad
        M = \frac{B^k - B}{B - 1}.
    \] Since $bM / B = b(B^{k - 1} - 1) / (B - 1)$ is an integer for $k > 1$, so is
    $bc_kB^{k - 1} / (bB - (B - 1)c_k)$.
\end{corollary}
\begin{proof}
    Recall that \[
        (aB + b)c = a(bB^k + c), \qquad
        a(bB^k + c - Bc) = bc.
    \] Since $b = c_{i < k}$ by Lemma~\ref{lemma:last_block}, we can write \[
        c = [bb\dots b\,c_k] = bB^{k - 1} + \dots + bB + c_k = 
        b\frac{B^k - B}{B - 1} + c_k = bM + c_k.
    \] Thus, \[
        a = \frac{bc}{bB^k - (B - 1)c} = \frac{b(bM + c_k)}{bB^k - (b(B^k - B) + (B - 1)c_k)} 
        = \frac{b(bM + c_k)}{bB - (B - 1)c_k}.
    \] Now, note that \[
        (bB - (B - 1)c_k)\frac{M}{B} = bM - (B^{k - 1} - 1)c_k = bM + c_k - B^{k - 1}c_k.
    \] Thus, \[
        \frac{bM + c_k}{bB - (B - 1)c_k} = \frac{M}{B} + \frac{c_kB^{k - 1}}{bB - (B - 1)c_k},
    \] whence \[
        a = \frac{b(bM + c_k)}{bB - (B - 1)c_k} = \frac{bM}{B} + 
        \frac{bc_kB^{k - 1}}{bB - (B - 1)c_k}. 
    \]
\end{proof}

The following observation regarding the final digit $c_k$ ties up the
proof of Theorem~\ref{theorem:structure}. This requires diving into the
prime factorisation of the base $B$.

\begin{lemma}\label{lemma:gcds}
    Let $B$ be an arbitrary integer base and let $N_k$ have property 
    $P^*_k$. Then, $\gcd(c_k, B) > 1$, i.e.\ the last digit $c_k$ must share some
    factor $p > 1$ with $B$.
    Furthermore, if $a_k \neq b$, then $gcd(a_k - b, B) > 1$.
\end{lemma}
\begin{proof}
    Suppose that $N_k$ has property $P^*_k$. Then $a, b, c > 0$, $c_k > 0$, and \[
        (aB + b)c = a(bB^k + c), \qquad
        (ac - abB^{k - 1})B = (a - b)c.
    \] This gives $B \mid (a - b)c$; writing \begin{align*}
        a' &= [a_1a_2\dots a_{k - 1}], & a &= a'B + a_k, \\
        c' &= [c_1c_2\dots c_{k - 1}], & c &= c'B + c_k,
    \end{align*} 
    we have $B \mid (a - b)(c'B + c_k)$, hence $B\mid (a - b)c_k = (a'B + a_k - b)c_k$,
    hence $B \mid (a_k - b)c_k$. Let \[
        B = p_1^{\alpha_1}p_2^{\alpha_2}\dots p_m^{\alpha_m}
    \] be the prime factorization of $B$, with each $\alpha_i \geq 1$.
    Then, for each prime factor $p_i$ of $B$, $p_i^{\alpha_i}\mid (a_k - b)c_k$.
    Let $\beta_i$ be the greatest integer such that $p_i^{\beta_i} \mid c_k$.
    Then, we must have $p_i^{\alpha_i - \beta_i} \mid a_k - b$; if not, the greatest
    power of $p_i$ dividing $(a_k - b)c_k$ would have been strictly less than $\beta_i + 
    (\alpha_i - \beta_i) = \alpha_i$, a contradiction.

    It is clear that we cannot have all $\beta_i \geq \alpha_i$; if so, we would have all
    $p_i^{\alpha_i} \mid c_k$, hence the product $p_1^{\alpha_1}\dots p_m^{\alpha_m} = B
    \mid c_k$, but $0 < c_k < B$, a contradiction. Thus, there must be some $\beta_j <
    \alpha_j$ corresponding to which $p^{\alpha_j - \beta_j} \mid a_k - b$, hence 
    $\gcd(a_k - b, B) \geq p_j > 1$. This proves the second part of our lemma.

    To prove the first part, we use induction on the block size $k$.
    Consider $k = 1$, where $a_k = a$, $c_k = c$, and suppose that all $\beta_i = 0$. 
    This forces all $p_i^{\alpha_i - \beta_i} = p_i^{\alpha_i} \mid a - b$, hence 
    their product $p_1^{\alpha_1}\dots p_m^{\alpha_m} = B \mid a - b$. 
    But $0 \leq \vert a - b\vert < B$, forcing $a - b = 0$.
    By Lemma~\ref{lemma:trivial_2}, this gives a trivial solution, a contradiction.
    Thus, there must be some $\beta_{j'} > 0$, hence $\gcd(c, B) \geq p_j^{\beta_j'} > 1$.

    Next, suppose that the statement holds for some $k \geq 1$, and let \[
        N = [a_1a_2\dots a_k a_{k + 1} \,b\, c_1c_2\dots c_{k + 1}] =
        [a\,b\,c]
    \] have property $P^*_{k + 1}$. Then, $a, b, c, > 0$, $c_{k + 1} > 0$,
    and $B \mid (a_{k + 1} - b)c_{k + 1}$. Again, if all $\beta_i = 0$,
    then we must have all $p_i^{\alpha_i - \beta_i} = p_i^{\alpha_i} \mid a_{k + 1} - b$,
    hence their product $B \mid a_{k + 1} - b$.
    Since $0 \leq \vert a_{k + 1} - b \vert < B$, we have $a_{k + 1} - b = 0$. But we also have
    $b = c_1 > c_{k + 1}$ by Lemma~\ref{lemma:last_block}.
    By Proposition~\ref{proposition:extension}, the number with the digits $a_{k + 1}, 
    c_1 = b$ removed, i.e.\ \[
        N' = [a_1a_2\dots a_k\, b\, c_2\dots c_{k + 1}],
    \] must have the $P^*_k$ property (each block remains non-zero, and $b > 
    c_{k + 1}$ so not all digits are equal). Applying our induction hypothesis, we have 
    $\gcd(c_{k + 1}, B) > 1$.

    Thus, by induction, our statement holds for all $k \geq 1$.
\end{proof}

\section{Estimate of saturation points.} \label{section:estimate_saturation}



We are now ready to prove Theorem~\ref{theorem:saturation_point} using
Corollary~\ref{corollary:ac_blocks}.

\begin{proof}[Proof of Theorem~\ref{theorem:saturation_point}]
    Suppose that the digits $b, c_k$ generate a solution of the $P^*_k$ problem
    as per Lemma~\ref{lemma:last_block}; further suppose that they do \emph{not}
    generate a solution for the $P^*_{k - 1}$ problem.
    Then, Corollary~\ref{corollary:ac_blocks} guarantees that \[
        \frac{bc_kB^{k - 1}}{bB - (B - 1)c_k}
    \] is an integer. Since $b, c_k$ do not generate a solution for the $P^*_{k - 1}$
    problem, either \[
        \frac{bc_kB^{k - 2}}{bB - (B - 1)c_k}
    \] is not an integer, or the first block $a'$ is too small, i.e.\ the block \[ 
        a' = \frac{b(B^{k - 2} - 1)}{B - 1} + \frac{bc_kB^{k - 2}}{bB - (B - 1)c_k} 
        < B^{k - 2}.
    \] The divisibility conditions in the first case are enough to obtain certain relations
    between $k$ and the prime factors of $B$ and $bc_k$.
    The bound on $k$ in the second case requires only direct algebraic manipulation. \\
    
    Consider the former case, and factorize \[
        B = p_1^{\alpha_1}\cdots p_r^{\alpha_r}, \qquad
        bc_k = p_1^{\beta_1}\cdots p_r^{\beta_r},
    \] where $p_1, \dots, p_r$ are primes, and each $\alpha_i, \beta_i \geq 0$. Then, the
    denominator \[
        bB - (B - 1)c_k \mid bc_kB^{k - 1} = p_1^{\alpha_1(k - 1) + \beta_1}\cdots
        p_r^{\alpha_r(k - 1) + \beta_r},
    \] hence its prime factorization cannot have any primes apart from $p_1,\dots, p_r$.
    Write \[
        bB - (B - 1)c_k = p_1^{\gamma_1}\cdots p_r^{\gamma_r}
    \] for $\gamma_i \geq 0$. On the other hand, the denominator \[
        bB - (B - 1)c_k \nmid bc_kB^{k - 2} = p_1^{\alpha_1(k - 2) + \beta_1}\cdots
        p_r^{\alpha_r(k - 2) + \beta_r}.
    \] This means that we must have some $j$ for which $\gamma_j > \alpha_j(k - 2) + 
    \beta_j$. Furthermore, $\gamma_j \leq \alpha_j(k - 1) + \beta_j$ so as to satisfy
    the first divisibility, ensuring that $\alpha_j \neq 0$. Now, \[
        (B - 1)^2 \geq (B - 1)(B - c_k) \geq bB - (B - 1)c_k \geq p_j^{\gamma_j} 
        > p_j^{\alpha_j(k - 2) + \beta_j},
    \] which gives \[
        (B - 1)^2 > p_j^{\alpha_j(k - 2) + \beta_j} \geq 2^{\alpha_j(k - 2)}, \qquad
        2\log_2(B - 1) > \alpha_j(k - 2).
    \] Since $\alpha_j \geq 1$, \[
        k - 2 < \frac{2\log_2(B - 1)}{\alpha_j} \leq 2\log_2(B - 1), \qquad
        k < 2\log_2(B - 1) + 2.
    \]
    
    Next, we consider the latter case where the block $a'$ for the $P^*_{k - 1}$ problem is
    too small, hence \[
        a' = \frac{b(B^{k - 2} - 1)}{B - 1} + \frac{bc_kB^{k - 2}}{bB - (B - 1)c_k} < 
        B^{k - 2}.
    \] However, the corresponding block $a$ for the $P^*_k$ problem is of the right size,
    hence \[
        a = \frac{b(B^{k - 1} - 1)}{B - 1} + \frac{bc_kB^{k - 1}}{bB - (B - 1)c_k} \geq
        B^{k - 1}.
    \] Collecting the powers of $B$, we have the equations \begin{align*}
        B^{k - 2}\left[\frac{b^2B}{(B - 1)(bB - (B - 1)c_k)} - 1\right] 
        &< \frac{b}{B - 1}, \\
        B^{k - 1}\left[\frac{b^2B}{(B - 1)(bB - (B - 1)c_k)} - 1\right] 
        &\geq \frac{b}{B - 1}.
    \end{align*} Note that the term inside the square bracket must be positive by the
    second equation. Thus, we can divide the first equation by this and obtain the 
    estimate \begin{align*}
        B^{k - 2} &< \frac{b}{B - 1}\cdot 
        \frac{(B - 1)(bB - (B - 1)c_k)}{b^2B - (B - 1)(bB - (B - 1)c_k)} \\
        &= \frac{b^2B - (B - 1)bc_k}{b^2B - bB(B -1) + (B - 1)^2c_k} \\
        &\leq b^2B - (B - 1)bc_k \\
        &\leq (B - 1)^3.
    \end{align*}
    Thus, \[
        (B - 1)^{k - 2} < B^{k - 2} \leq (B - 1)^3, \qquad
        k - 2 < 3, \qquad
        k < 5.
    \] This means that new solutions cannot appear with increasing $k$,
    when \[
        k \geq 2\log_2(B - 1) + 2,\quad \text{ and }\quad k \geq 5,
    \] that is, \[
        k \geq \max\{5,\, 2\log_2(B - 1) + 2\}. 
    \]
\end{proof}





\section{Powers of primes.}

In this section, we examine solutions of $P^*_k$ where the base $B = p^n$ for
prime $p$, $n \geq 2$.

\begin{lemma} \label{lemma:pn_extension}
    Let $B = p^n$ where $p$ is prime, $n > 1$. If $N_k$ has property $P^*_k$ for 
    $k \geq 1$, then $a_{i \neq 1} = b = c_{i \neq k}$. Furthermore, $p \mid c_k$.
\end{lemma}
\begin{proof}
    Suppose that $N_k$ has property $P^*_k$ for $k > 1$. Then, $a, b, c > 0$ via
    Lemma~\ref{lemma:trivial_1}, the last digit $c_k > 0$, and $p\mid c_k$ using 
    Lemma~\ref{lemma:gcds}. By Corollary~\ref{corollary:ac_blocks}, write \[
        (a, b, c) = \left(\frac{bM}{B} + \frac{bc_kB^{k - 1}}{bB - (B - 1)c_k}, 
        b, bM + c_k\right), \qquad
        M = \frac{B^k - B}{B - 1}.
    \] Specifically, \[
        \frac{bc_k B^{k - 1}}{bB - (B - 1)c_k} = q
    \] is an integer. Write $B = p^n$, and $c_k = p^rc_k'$, $q = p^sq'$, $b = p^t b'$ with 
    $p\nmid c_k', q', b'$. Since $c_k < B = p^n$ is a single digit, we must have $r < n$.
    Now, we have \[
        bc_k p^{n(k - 1)} = q(bp^n - (p^n - 1)c_k) = q((b - c_k)p^n + c_k),
    \] hence \[
        bc_k' p^rp^{n(k - 1)} = q((b - c_k)p^n + p^rc_k'), \qquad
        bc_k' p^{n(k - 1)} = q((b - c_k)p^{n - r} + c_k').
    \] Now, \[
        b'c_k'p^t p^{n(k - 1)} = q'p^s((b - c_k)p^{n - r} + c_k'),
    \] hence \[
        b'c_k' p^{n(k - 1) + t - s} = q'(b - c_k)p^{n - r} + q'c_k'.
    \] Note that we have integers on both sides. Since $r < n$, we have 
    $p\mid q'(b - c_k)p^{n - r}$; but by construction, $p\nmid q'c_k'$. Thus, the left
    hand side $p\nmid b'c_k' p^{n(k - 1) + t - s}$. Again, $p \nmid b'c_k'$, hence we have
    $s = n(k - 1) + t$. Thus, \[
        \frac{bc_k}{bB - (B - 1)c_k} = \frac{q}{p^{n(k - 1)}} = q' p^{s - n(k - 1)} = q' p^t
    \] is an integer.

    Now, the number $N_* = [a_*\,b_*\,c_*]$ where \[
        (a_*, b_*, c_*) = \left(\frac{bc_k}{bB - (B - 1)c_k}, b, c_k\right)
    \] has property $P^*_1$. To see this, note that $c_k < b$ gives \[
        a_* = \frac{bc_k}{bB - (B - 1)c_k} < \frac{bc_k}{bB - (B - 1)b} = c_k < B
    \] ensuring that $a_*$ is a single digit, and that \[
        \frac{1}{a_*} + \frac{B - 1}{b_*} = \frac{bB - (B - 1)c_k}{bc_k} + \frac{B - 1}{b} = 
        \frac{B}{c_k} = \frac{B^1}{c_*},
    \] satisfying the $P_k$ property.

     By Proposition~\ref{proposition:extension}, its extension \[
        N_*^+ = [a_*b_*\dots b_*\, b_*\,b_*\dots b_*c] = 
        [a_*b\dots b\, b\, b\dots b c_k] = [a_*^+\, b\, c]
    \] has property $P^*_k$. However, the digits $b, c_k$ uniquely determine the 
    first block $a_*^+$, and we already have a solution $N = [a\, b\, c]$ generated 
    by $b, c_k$. This forces $a = a_*^+$, hence all $a_{i \neq 1} = b$ as desired.
    In other words, our original solution $N$ for $P_k$ is an extension of the 
    solution $N_*$ for $P_1$.
\end{proof}

\end{appendices}


\end{document}